\documentclass[12pt]{amsart}

\usepackage{amsmath,amsfonts,amssymb,amscd,amsthm,latexsym}

\usepackage{pstricks,pst-node,pst-plot}
\usepackage{graphicx}

\theoremstyle{plain}

\newtheorem{theorem}{Theorem}[section]
\newtheorem{proposition}[theorem]{Proposition}

\newtheorem*{theorem*}{Theorem}
\newtheorem{corollary}[theorem]{Corollary}
\newtheorem*{corollary*}{Corollary}
\newtheorem{lemma}[theorem]{Lemma}

\theoremstyle{example}

\newtheorem{remark}[theorem]{Remark}

\begin{document}

\vspace{10mm}
\begin{flushright}
{\em Dedicated to Yury Evgen'evich Shishmaryov}
\end{flushright}

\vspace{5mm}

\title[Abelian and Hamiltonian groupoids] {Abelian and Hamiltonian groupoids} \subjclass{20 M 10}

\noindent УДК~510.8:512.57 
\date{\today}
\author{Stepanova A.A.${}^1$, Trikashnaya N.V.} \footnotetext[1]
{This research was partially supported by the grant of the leading
science schools of Russia (grant SS-2810.2008.1) and by RFBR
(grant 09-01-00336-a)}
\address{Institute of Mathematics and Computer Science\\Far East State
University\\Vladivostok\\Russia} \email{stepltd@mail.ru}
\address{Institute of Mathematics and Computer Science\\Far East State
University\\Vladivostok\\Russia} \email{trik74@mail.ru}

\begin {abstract} In the work we investigate some
groupoids which are the Abelian algebras and the Hamiltonian
algebras. An algebra is Abelian if for every polynomial operation
and for all elements $a,b,\bar c,\bar d$ the implication $t(a,\bar
c)=t(a,\bar d)\rightarrow t(b,\bar c)=t(b,\bar d)$ holds; an
algebra is Hamiltonian if every subalgebra is a block of some
congruence on the algebra. R.V. Warne in 1994 described the
structure of the Abelian semigroups. In this work we describe the
Abelian groupoids with identity, the Abelian finite quasigroups
and the Abelian semigroups $S$ such that $abS=aS$ and $Sba=Sa$ for
all $a,b\in S$. We prove that a finite Abelian quasigroup is a
Hamiltonian algebra. We characterize the Hamiltonian groupoids
with identity and semigroups under the condition of Abelian of
this algebras.
\end {abstract} \maketitle

\sloppy \vskip 1cm

\section {Introduction}

The Abelian and Hamiltonian properties for algebras were
investigated in \cite{hm,kv,kval}. In this work we study the
groupoids, which are Abelian algebras and Hamiltonian algebras.

Let us remind some definitions. An algebra $\langle
A;\cdot\rangle$ with binary operation $ \cdot $ is called a
groupoid. A groupoid $\langle A;\cdot\rangle$ is called a
quasigroup if for any $a, b\in A $ there exist uniquely determined
elements $x,y\in A$ satisfying $x\cdot a=b $, $a\cdot x=b $. A
quasigroup $\langle A;\cdot\rangle$ with identity element $1$ such
that $1\cdot a=a\cdot 1=a$ for every element $a\in A$ is called a
loop. A groupoid with an operation satisfying the associative law
is called a semigroup.

An algebra is Abelian if for every polynimial operation $t (x,
y_1, \ldots, y_n) $ and for all elements $u, v, c_1, \ldots, c_n,
d_1, \ldots, d_n $ in algebra the equality $t(u, c_1, \ldots, c_n)
=t (u, d_1, \ldots, d_n)$ implies $t (v, c_1, \ldots, c_n) =t (v,
d_1, \ldots, d_n).$ An algebra is Hamiltonian if every subalgebra
is a block of some congruence on the algebra. It is not hard to
show that any group $\langle A;\cdot\rangle$ is Abelian iff it is
commutative; it is Hamiltonian iff every subgroup is a normal
subgroup; as easy to prove is that any module or essentially unary
algebra is Abelian and Hamiltonian. In \cite {ovch} there is the
description of Abelian groupoid $\langle A;\cdot\rangle$ with
$|A\cdot A|\leq 3$. In \cite{war1,war2} it is characterized the
Abelian semigroups, the periodic Abelian semigroups, the
semisimple Abelian semigroups; it is considered the questions,
which connect with the Hamiltonian semigroups.

\section {Groupoids with identity}

In this section we consider the groupoids with the identity
satisfying the Abelian and Hamiltonian properties. The identity of
groupoid we  denote by 1.

\begin {theorem} \label {ab gr s 1}
Let $\langle A; \cdot\rangle$ be a groupoid with the identity. A
groupoid $\langle A;\cdot\rangle$ is an Abelian algebra iff
$\langle A;\cdot\rangle$ is a commutative semigroup such that for
all $a, b\in A$ the equation $a\cdot x=b$ has not more than one
solution in $\langle A;\cdot\rangle$.
\end {theorem}

\begin {proof} Let $\langle A;\cdot\rangle$ be Abelian algebra. We will show that
$\langle A; \cdot\rangle$ is a semigroup. Let $a, b, c\in A$.
Since $\langle A;\cdot\rangle$ is a groupoid with identity then
$(1\cdot b)\cdot(c\cdot 1)=(1\cdot1)\cdot(b\cdot c)$. As $\langle
A;\cdot\rangle$ is an Abelian algebra then $(a\cdot b)
\cdot(c\cdot 1)=(a\cdot 1)\cdot(b\cdot c)$. Hence $(a\cdot b)
\cdot c=a\cdot(b\cdot c)$ and the associative law in a groupoid
$\langle A;\cdot\rangle$ holds.

We will show that $\langle A;\cdot\rangle$ is a commutative
semigroup. Let $a, b\in A$. Since an algebra $\langle
A;\cdot\rangle$ is Abelian then the equality $1\cdot 1\cdot a =
a\cdot 1\cdot 1$ implies $1\cdot b\cdot a = a\cdot b\cdot 1$. Thus
$a\cdot b = b\cdot a $ and the commutative law in a semigroup
$\langle A;\cdot\rangle$ holds.

We will show that for any $a, b\in A$ the equation $a\cdot x = b$
has not more than one solution. Assume that $a\cdot c_1 = a\cdot
c_2$ for some $c_1, c_2\in A$. As $\langle A;\cdot\rangle$ is an
Abelian algebra then $1\cdot c_1 = 1\cdot c_2$. Hence $c_1 = c_2$.

Let us prove sufficiency. Let $t (x, y_1, \ldots, y_n)$ be a
polynomial operation of an algebra $\langle A;\cdot\rangle$. Since
$\langle A;\cdot\rangle$ is a commutative semigroup then $t (x,
y_1, \ldots, y_n) = x^k \cdot y_1 ^ {k_1}\cdot \ldots \cdot y_n ^
{k_n}$ for some $k, k_1, \ldots, k_n\in \omega$. Assume $a_1,
\ldots, a_n, b_1, \ldots, b_n, c, e\in A $ and $c^k \cdot (a_1 ^
{k_1} \cdot \ldots \cdot a_n ^ {k_n}) = c^k \cdot (b_1 ^ {k_1}
\cdot \ldots \cdot b_n ^ {k_n}) =e$. As the equation $c^k\cdot x =
e $ has the unique solution then $a_1 ^ {k_1} \cdot \ldots \cdot
a_n ^ {k_n} =b_1 ^ {k_1} \cdot \ldots \cdot b_n ^ {k_n}$. Hence
$d^k \cdot (a_1 ^ {k_1} \cdot \ldots \cdot a_n ^ {k_n}) = d^k
\cdot (b_1 ^ {k_1} \cdot \ldots \cdot b_n ^ {k_n}) $ for any $d\in
A $. Thus $\langle A;\cdot\rangle$ is an Abelian algebra.
\end {proof}

\begin {corollary}
A finite groupoid with an identity is an Abelian algebra iff it is
an Abelian group.
\end {corollary}

\begin {corollary}
A finite Abelian groupoid with an identity is a Hamiltonian
algebra.
\end {corollary}

\begin {lemma} \label {gam polugr period}
If a semigroup $\langle A;\cdot\rangle$ is a Hamiltonian algebra
then for any $a\in A $ there exist $i, j $, $1\leq i <j $, such
that $a^i=a^j$.
\end {lemma}

\begin {proof}
Let $\langle A;\cdot\rangle$ be Hamiltonian semigroup and $a\in A
$. Assume that $a ^ {i} \neq a ^ {j}$ for any $i, j $, $i\neq j $.
Let $B = \{a ^ {2} \} \cup \{a ^ {k} \mid k\geq 4 \}$. Then
$\langle B,\cdot\rangle$ be a subsemigroup of a semigroup $\langle
A,\cdot\rangle$. As $\langle A,\cdot\rangle$ is Hamiltonian then
there exists a congruence $\Theta$ such that $B$ is its class.
From $a\Theta a $ and $a ^ {2} \Theta a ^ {4}$ follows $a ^ {3}
\Theta a ^ {5}$, i.e. $a ^ {3} \in B$. Contradiction.
\end {proof}

\begin {theorem}
Let $\langle A;\cdot\rangle$ be Abelian groupoid with identity.
The groupoid $\langle A;\cdot\rangle$ is a Hamiltonian algebra iff
$\langle A;\cdot\rangle$ is a periodic Abelian group.
\end {theorem}

\begin {proof} Suppose $\langle A;\cdot\rangle$ is an Abelian Hamiltonian groupoid with
the identity. By Theorem \ref{ab gr s 1} $\langle A;\cdot\rangle$
is a semigroup. Let $a\in A $. By Lemma \ref {gam polugr period}
there exist $i, j $, $1\leq i <j $, such that $a^i=a^j $, i.e.
$1\cdot a^i=a ^ {j-i} \cdot a^i $. As the semigroup $\langle
A;\cdot\rangle$ is an Abelian algebra then $1\cdot 1=a ^ {j-i}
\cdot 1 $, i.e. $1=a ^ {j-i}$. In particular for any $a, b\in A $
the equation $ax=b $ has the solution. Then by Theorem \ref{ab gr
s 1} and by Lemma \ref {gam polugr period} $\langle
A;\cdot\rangle$ is a periodic Abelian group.

Assume that $\langle A;\cdot\rangle$ is a periodic Abelian group.
Then any subgroupoid $\langle A;\cdot\rangle$ contains an
identity, that is it is a subgroup, and any its subgroup is a
normal subgroup, i.e. a class of some congruence of the groupoid
$\langle A;\cdot\rangle$.
\end {proof}

\section {Quasigroups}

Suppose $\langle A;\cdot\rangle$ is a quasigroup, $a\in A $. Let
us put (see \cite {belous})
 $$ R_a (x) = x \cdot a, \; \; L_a (x) = a \cdot x,
\; \; x + y = R_a ^ {-1} (x) \cdot L_a ^ {-1} (y). $$
 It is clear that $ R_a (x) $ and $L_a (x) $ are the permutations of a set $A $ and $
\langle A; + \rangle $ is a quasigroup.

\begin {remark} \label {lupa}
\cite {belous} Let $\langle A;\cdot\rangle$ be a quasigroup, $a\in
A $. Then

1) $ \langle A; + \rangle $ is a loop with an identity element $0
= a \cdot a $;

2) the equalities $r_a (x) + a = R_a ^ {-1} (x) $ and $l_a (x) + a
= L_a ^ {-1} (x) $ define the permutations $r_a (x)$ and $l_a (x)$
of a set $A $;

3) $x \cdot y = R_a (x) + L_a (y) $ for any $x, y\in A $;

4) $R_a (a) =L_a (a) =0 $;

5) $R_a ^ {-1} (0) =L_a ^ {-1} (0) = a $.
\end {remark}

\begin {lemma} \label {R L +} Let $\langle A;\cdot\rangle$ be a finite quasigroup, $a\in
A $. Then the functions $R_a ^ {-1} (x) $, $L_a ^ {-1} (x) $ and
$x+y $ on a set $A $ are determined by the polynomial operations
on an algebra $\langle A;\cdot\rangle$.
\end {lemma}

\begin {proof}
Suppose $\langle A;\cdot\rangle$ is a finite quasigroup and $a \in
A $. Let $g_a^k (x) = (\ldots ((x \cdot a) \cdot a) \cdot\ldots)
\cdot a$ where $a$ occurs exactly $k $ times. Since a set $
\{g_a^0 (x), g_a^1 (x), \ldots, g_a^k (x), \ldots \}$ is finite
for all $x\in A $ then there exists $n> 0 $ (which does not
depends on $x $) such that $g_a^n (x) =g_a ^ {n-1} (x) \cdot a = x
$. As $R_a ^ {-1} (x) \cdot a = x $ then $g_a ^ {n-1} (x) = R_a ^
{-1} (x) $. Similarly, $h_a ^ {k-1} (y) = L_a ^ {-1} (y) $ for
some polynomial operation $h_a ^ {k-1} (y) $ on an algebra
$\langle A;\cdot\rangle$. Therefore $x + y = R_a ^ {-1} (x) \cdot
L_a ^ {-1} (y) = g_a ^ {n-1} (x) \cdot h_a ^ {k-1} (y) $, that is
a function $x + y $ is determined by a polynomial operation $g_a ^
{n-1} (x) \cdot h_a ^ {k-1} (y) $ on an algebra $\langle
A;\cdot\rangle$.
\end {proof}

\begin {theorem} \label {abel kvasigr} Let $\langle A;\cdot\rangle$ be a finite quasigroup, $a\in A $.
The quasigroup $\langle A;\cdot\rangle$ is an Abelian algebra iff

1) $ \langle A; + \rangle $ is an Abelian group,

2) the permutations $r_a (x) $ and $l_a (x) $ are the
automorphisms of $ \langle A; + \rangle $.
\end {theorem}

\begin {proof}
Suppose $\langle A;\cdot\rangle$ is a finite Abelian quasigroup
and $a \in A $. By Lemma \ref {R L +} the functions $R_a ^ {-1}
(x) $, $L_a ^ {-1} (x) $ and $x+y $ are determined by the
polynomial operations on an algebra $\langle A;\cdot\rangle$. So
an algebra $ \langle A; + \rangle $ is Abelian. By Remark \ref
{lupa}$ \langle A; + \rangle $ is a loop with an identity element
0. Then by Theorem \ref{ab gr s 1}$ \langle A; + \rangle $ is an
Abelian group.

Let us show that the permutation $r_a (x) $ is an automorphism of
an Abelian group $ \langle A; + \rangle $. It is enough to proof
that $r_a (x) $ is a homomorphism of this group. Let $b, c\in A $.
In Abelian group $ \langle A; + \rangle $ equality
 $$ R_a ^ {-1} (0+c) +R_a ^ {-1} (0) =R_a ^ {-1} (0+0) +R_a ^ {-1} (c) $$
 holds.
Since an algebra $\langle A;\cdot\rangle$ is Abelian and the
functions $R_a ^ {-1} (x) $ and $x+y $ are determined by the
polynomial operations on an algebra $\langle A;\cdot\rangle$ then
 $$ R_a ^ {-1} (b+c) +R_a ^ {-1} (0) =R_a ^ {-1} (b+0) +R_a ^ {-1} (c.) $$
By definition of the permutation $r_a (x) $ and by Remark \ref
{lupa}
 $$ (r_a (b+c) +a) +a = (r_a (b) +a) + (r_a (c) +a), $$
 $$ r_a (b+c) =r_a (b) +r_a (c). $$
 Similarly $l_a (x)$ is an automorphism of Abelian
 group $ \langle A; + \rangle $.

Let us prove sufficiency. By $R $ denote a ring of endomorphisms
of a groups $ \langle A; + \rangle $, generated by the
automorphisms $r_a^1 $ and $l_a^1 $. We will show that for any
polynomial operation $f (x_0, \ldots, x_n) $ on an algebra
$\langle A;\cdot\rangle$ there exist $d\in A $ and $ \alpha_1,
\ldots, \alpha_n\in R $ such that for any $x_0, \ldots, x_n\in A $
the following holds:
$$ f (x_0, \ldots, x_n) = \alpha_0 (x_0) + \ldots + \alpha_n (x_n)
+ d.\eqno (1) $$ An induction on the complexity of $f (x_0,
\ldots, x_n) $. If $f (x_0, \ldots, x_n) =x_0 $ than (1) is
obvious. Let $f (x_0, \ldots, x_n) = g (x_0, \ldots, x_n) \cdot h
(x_0, \ldots, x_n) $. By the induction hypothesis
$$ g (x_0,
\ldots, x_n) = \beta_0 (x_0) + \ldots + \beta_n (x_n) + b, $$
$$ h (x_0,
\ldots, x_n) = \gamma_0 (x_0) + \ldots + \gamma_n (x_n) + c, $$
where $ \beta_i, \gamma_i\in R $. Then using Remark \ref {lupa} we
get
$$ f (x_0,
\ldots, x_n) =g (x_0, \ldots, x_n) \cdot h (x_0, \ldots, x_n) = $$
$$ =R_a (g (x_0, \ldots, x_n)) +L_a (h (x_0, \ldots, x_n)) = $$
$$ =r_a ^ {-1} (g (x_0, \ldots, x_n)-a) +l_a ^ {-1} (h (x_0, \ldots,
x_n)-a) = $$ $$ =r_a ^ {-1} (g (x_1, \ldots, x_n))-r_a ^ {-1} (a)
+l_a ^ {-1} (h (x_0, \ldots, x_n))-l_a ^ {-1} (a) = $$
$$ =r_a ^ {-1} (\beta_0 (x_0) + \ldots + \beta_n (x_n) +
b) +l_a ^ {-1} (\gamma_0 (x_0) + \ldots + \gamma_n (x_n) + c) +r =
$$
$$ =r_a ^ {-1} (\beta_0 (x_0)) + \ldots + r_a ^ {-1} (\beta_n (x_n))
+l_a ^ {-1} (\gamma_0 (x_0)) + \ldots + l_a ^ {-1} (\gamma_n
(x_n)) + d = $$
$$ = (r_a ^ {-1} \beta_0+l_a ^ {-1} \gamma_0) (x_0) + \ldots +
(r_a ^ {-1} \beta_n+l_a ^ {-1} \gamma_n) (x_n) + d $$
 for some
$r, d\in A $. Clearly $r_a ^ {-1} \beta_i+l_a ^ {-1} \gamma_i\in R
$. Thus, (1) it is proved.

Let us show that $\langle A;\cdot\rangle$ is an Abelian algebra.
Suppose $f (x_0, x_1, \ldots, x_n) $ is a polynomial operation on
this algebra and $f (b, r_1, \ldots, r_n) = f (b, s_1, \ldots,
s_n) $, where $b, r_1, \ldots, r_n, s_1, \ldots, s_n\in A $. Using
(1) we receive
 $$\alpha_0 (b) + \alpha_1 (r_1) + \ldots + \alpha_n
(r_n) + d =\alpha_0 (b) + \alpha_1 (s_1) + \ldots + \alpha_n (s_n)
+ d. $$ Then
$$\alpha_1 (r_1) + \ldots +
\alpha_n (r_n) + d = \alpha_1 (s_1) + \ldots + \alpha_n (s_n) + d.
$$ Hence
$$\alpha_0 (c) + \alpha_1 (r_1) + \ldots +
\alpha_n (r_n) + d =\alpha_0 (c) + \alpha_1 (s_1) + \ldots +
\alpha_n (s_n) + d $$ for any $c\in A $. Thus $\langle
A;\cdot\rangle$ is an Abelian algebra.
\end {proof}

The following proposition gives us some necessary condition for a
finite quasigroup to be Abelian. This condition will be used for
construction the examples in this section.

\begin {proposition} \label {diagon ab kvaz} If $\langle A;\cdot\rangle$
is an Abelian quasigroup then there exists $n\in \omega $ such
that either $ \{x\in A \mid x^2=a \} = \emptyset $ or $ | \{x\in
A\mid x^2=a \} | = n $ for any $a\in A $.
\end {proposition}

\begin {proof} Let $a, b, a_i, b_j\in A $ such that $a_i^2=a $,
$b_j^2=b $, where $1\leq i\leq n $, $1\leq j\leq m $ and $n\geq m
$. Choose $c\in A $. Since $\langle A;\cdot\rangle$ is a
quasigroup then there exist distinct $d_1\ldots, d_n\in A $ such
that $a_i=c\cdot d_i $ for all $i $, $1\leq i\leq n $. Then
$$ (c
\cdot d_i) \cdot (c \cdot d_i) = (c \cdot d_j) \cdot (c \cdot d_j)
$$ for any $i, j $, $1\leq i\leq n $, $1\leq j\leq m $. Let
$r\in A $ be an element such that $b_1=r\cdot d_1 $. By Abelian
property for $\langle A;\cdot\rangle$ we have
$$ (r
\cdot d_1) \cdot (r \cdot d_1) = (r \cdot d_j) \cdot (r \cdot
d_j), $$ that is $b = (r\cdot d_j) ^2 $ for any $j $, $1\leq j\leq
m $. Since the elements $rd_1\ldots, rd_n $ are distinct then
$n\leq m $. The proposition is proved.
\end {proof}

\begin {theorem} \label {gam kvasigr} Every finite Abelian
quasigroup is a Hamiltonian algebra.
\end {theorem}

\begin {proof} Let $\langle A;\cdot\rangle$
be a finite Abelian quasigroup, $ \langle B; \cdot\rangle $ be a
subalgebra and $a\in B $. By Theorem \ref{abel kvasigr} $ \langle
A; + \rangle $ is an Abelian group with a identity element
$0=a\cdot a $. It is clear that $0\in B $. By Lemma \ref {R L +}
the functions $R_a ^ {-1} (x) $, $L_a ^ {-1} (x) $ and $x+y $ on a
set $A $ are determined by the polynomial operations on an algebra
$\langle A;\cdot\rangle$ moreover it follows from the proof of the
Lemma that this operations depend on a unique element of a set $A
$, an element $a $. Hence a set $B $ is closed under the
operations $R_a ^ {-1} (x) $, $L_a ^ {-1} (x) $, $x+y $ and $
\langle B; + \rangle $ is a subgroup of a group $ \langle A; +
\rangle $.

We claim that $ \langle B; \cdot\rangle $ is a quasigroup. Note
that the equations $x\cdot a=c $ and $a\cdot x=c $, where $c\in B
$, have solutions $R_a ^ {-1} (c) \in B $ and $L_a ^ {-1} (c) \in
B $ accordingly. The equation $x\cdot b=c $, where $b, c\in B $,
is equivalent to the equation $R_a (x) +L_a (b) =c $, that is
$x\cdot a+a\cdot b=c $. In a group $ \langle B; + \rangle $ there
is an element $d $ such that $d+a\cdot b=c $. The equation $x\cdot
a=d $ is solvable in $ \langle B; \cdot\rangle $. Hence the
equation $x\cdot b=c $ is solvable too in $ \langle B;
\cdot\rangle $. Similarly, the equation $b\cdot x=c $ is solvable
in $ \langle B; \cdot\rangle $, that is $ \langle B; \cdot\rangle
$ is a quasigroup.

We claim that a partition of a group $ \langle A; + \rangle $ into
the cosets  of a subgroup $ \langle B; + \rangle $ defines the
congruence on a quasigroup $\langle A;\cdot\rangle$. Assume that
$b\in B $ and $c\in A $. There exists $d\in B $ such that
$b=a\cdot d $. Then $ (c\cdot a) +b = (c\cdot a) + (a\cdot d)
=c\cdot d\in c\cdot B $. Moreover $c\cdot b = (c\cdot a) + (a\cdot
b) \in (c\cdot a) +B $. Hence $c\cdot B = (c\cdot a) +B $.
Similarly $B\cdot c = (a\cdot c) +B $. Therefore for any $c\in A $
there is $d\in A $ such that $c\cdot B=B\cdot d $. By Theorem
\ref{abel kvasigr} the permutations $r_a (x) $ and $l_a (x) $ are
the automorphisms of a group $ \langle A; + \rangle $. As  $B $ is
closed under the operations $R_a ^ {-1} (x) $, $L_a ^ {-1} (x) $
and from the definition of the permutations $r_a (x) $ and $l_a
(x) $ we obtain $r_a (B) =B $ and $l_a (B) =B $. Hence the
equality $ (c\cdot a) +B = (d\cdot a) +B $ is equivalent to the
equalities $r_a (c\cdot a) +B=r_a (d\cdot a) +B $, $ (c-a) +B =
(d-a) +B $ and $c+B=d+B $. Similarly the equality $ (a\cdot c) +B
= (a\cdot d) +B $ is equivalent to the equality $c+B=d+B $. Let
$c+B=d+B $ and $c '+B=d' +B $. Then $ (c\cdot a) +B = (d\cdot a)
+B $ and $ (a\cdot c ') +B = (a\cdot d ') +B $. So $ ((c\cdot a) +
(a\cdot c ')) +B = ((d\cdot a) + (a\cdot d ')) +B $, i.e. $
(c\cdot c ') +B = (d\cdot d ') +B $. Thus the partition of a group
$ \langle A; + \rangle $ into the cosets  of a subgroup $ \langle
B; + \rangle $ defines the congruence on a quasigroup $\langle
A;\cdot\rangle$ and $\langle A;\cdot\rangle$ is a Hamiltonian
algebra.
\end {proof}

The following \emph {example} shows that a condition  \emph {2)}
of Theorem \ref {abel kvasigr} is essential. Let a quasigroup $
\langle Q; \cdot\rangle $ is defined by the Cayley table:

\begin {center}

\begin{tabular}{c|cccc}
    % after \\: \hline or \cline{col1-col2} \cline{col3-col4} ...
  $\cdot$ & 0 & 1 & 2 & 3 \\
  \hline
  0 & 1 & 3 & 2 & 0 \\
  1 & 2 & 0 & 3 & 1 \\
  2 & 0 & 2 & 1 & 3 \\
  3 & 3 & 1 & 0 & 2 \\

\end{tabular}
\end {center}
Let us construct a loop $ \langle Q; +\rangle $ chosen as $a $ an
element $1\in Q $:

\begin {center}

\begin{tabular}{c|cccc}
    % after \\: \hline or \cline{col1-col2} \cline{col3-col4} ...
  + & 0 & 1 & 2 & 3 \\
  \hline
  0 & 0 & 1 & 2 & 3 \\
  1 & 1 & 2 & 3 & 0 \\
  2 & 2 & 3 & 0 & 1 \\
  3 & 3 & 0 & 1 & 2 \\

\end{tabular}
\end {center}
Then $ \langle Q; +\rangle $ is the residue class group modulo 4,
i.e. it is an Abelian group, and the permutation $r_1 (x) $ is not
an automorphism of this group: $r_1 (1) =R_1 ^ {-1} (1)-1=3-1=2 $,
$r_1 (1) +r_1 (1 =0 $, $r_1 (1+1) =r_1 (2) =R_1 ^ {-1} (2)-1=2-1=1
$. By Proposition \ref {diagon ab kvaz} a quasigroup $ \langle Q;
\cdot\rangle $ is not Abelian. Notice that a quasigroup $ \langle
Q; \cdot\rangle $ is not Hamiltonian too. Really, $ \langle \{0,1
\};\cdot\rangle $ is a subalgebra of $ \langle Q; \cdot\rangle $,
$2\cdot \{0,1 \} = \{0,2 \}$, so a subalgebra $ \langle \{0,1
\};\cdot\rangle $ is not a block of congruence of a quasigroup $
\langle Q; \cdot\rangle $.

A groupoid
\begin {center}

\begin{tabular}{c|cccc}
    % after \\: \hline or \cline{col1-col2} \cline{col3-col4} ...
  $\cdot$ & 0 & 1 & 2 & 3 \\
  \hline
  0 & 1 & 0 & 3 & 2 \\
  1 & 2 & 1 & 0 & 3 \\
  2 & 0 & 3 & 2 & 1 \\
  3 & 3 & 2 & 1 & 0 \\

\end{tabular}
\end {center}
is \emph {an example} of a grouppoid which is not Abelian (by
Proposition \ref {diagon ab kvaz}) and Hamiltonian (there is no
the more then one-element proper subalgeras) quasigroups.

\section {Semigroups}

Abelian and Hamiltonian semigroups are studied in this section. In
this section we will usually write $ab$ as alternatives to $a\cdot
b$, where $a $ and $b$ are the elements of semigroup $\langle
A,\cdot \rangle$.

The semigroup $\langle A,\cdot\rangle$ is called stationary \cite
{kp} if the equality $ub=uc $ implies $vb=vc $ and the equality
$bu=cu $ implies $bv=cv $ for all $u, v, b, c\in A $.

The following proposition we get directly from the definition of
Abelian algebra

\begin {proposition}\label{Razdu - otn equiv} A semigroup $\langle A,\cdot\rangle$ is Abelian
iff $\langle A,\cdot\rangle$ is stationary and for all $a, b, c,
d, u, v\in A $ the equality $aub=cud $ implies $avb=cvd $.
\end {proposition}

In \cite{war1} there is the characterization of the Abelian
semigroups. In case when semigroup $\langle A,\cdot\rangle$
satisfies the condition
$$\forall b, c \, \in A \; \; (bcA=bA \;\;\text {and} \;\; Abc=Ac)
\;\;\text {or set} \ A\cdot A \;\;\text {is finity} \eqno (*)
$$
it is possible to give a more evident description of a structure
of Abelian semigroups.

The following definitions can be found in \cite {kp}. The
semigroup $\langle A,\cdot\rangle$ is called a rectangular band of
semigroups, if there is a set $ \{A _ {i\lambda} \mid i \in I,
\lambda \in \Lambda \}$, which is a partition of set $A $, and $
\langle A _ {i\lambda}, \cdot \rangle $ are subsemigroups of
semigroup $\langle A,\cdot\rangle$ and for all $i \in I, \lambda,
\mu \in \Lambda $ inclusion $A _ {i\lambda} \cdot A _ {j\mu}
\subseteq A _ {i\mu}$ holds. The semigroup $\langle
A,\cdot\rangle$ is called an inflation of semigroup $ \langle B,
\cdot \rangle $ if there is a partition $ \{X _ {a} \mid a\in B
\}$ of set $A $ such that $a\in X _ {a}$ and $x\cdot y = a\cdot b
$ for all $a, b \in B, x \in X _ {a}, y\in X _ {b}$.

Let us define the equivalence relation on set $A $:

$$ a\alpha b \Leftrightarrow \forall x\in A (ax=bx \;\;\text {and} \;\; xa=xb) $$ for
all $a, b\in A $.

\begin {remark} \label{Has inflat - rel of equival} A semigroup
$\langle A,\cdot\rangle$ is an inflation of a semigroup $ \langle
B, \cdot \rangle $ iff there exists $b\in B $ such that $a\alpha b
$ for all $a\in A $.
\end {remark}

\begin {lemma} \label{if a inflated then aa is rectangl ab gr}
If a semigroup $\langle A,\cdot\rangle$ is an inflation of a
rectangular band $ \langle T, \cdot \rangle $ of Abelian groups
then $T=A\cdot A $.
\end {lemma}

\begin {proof} Let $\langle A,\cdot\rangle$ be an inflation of a rectangular band $T$
of Abelian groups. It is clear that $T\subseteq A\cdot A $. If $a
_ {1} \cdot a _ {2} \in A\cdot A $, then by definition of an
inflation of semigroup there exist $x, y\in T $ such that $x\alpha
a _ {1}$ and $y\alpha a _ {2}$, that is $a _ {1} \cdot a _ {2}
=x\cdot y \in T $ and $A\cdot A \subseteq T $.
\end {proof}

Let $\langle A,\cdot\rangle$ be a semigroup. We will introduce the
following relations on a set $A $:
 $$ x \Phi y \Leftrightarrow \exists
z (xz=yz), $$
 $$ x \Psi y\Leftrightarrow \exists z (zx=zy), $$
 $$ x X y \Leftrightarrow \exists z (zx=x
\wedge zy=y \wedge z ^ {2} =z), $$
 $$ x Y y \Leftrightarrow \exists z
(xz=x \wedge yz=y \wedge z ^ {2} =z), $$ $$ x Z y \Leftrightarrow
x X y \wedge x Y y. $$

\begin {remark} If a semigroup $\langle A,\cdot\rangle$ is Abelian
then
$$ x \Phi y \Leftrightarrow \forall
z (xz=yz), $$
$$ x \Psi y\Leftrightarrow \forall z (zx=zy) $$
and relations $ \Phi $ and $ \Psi $ are equivalence relations on a
set $A $.
\end {remark}

If $ \Theta $ is the equivalence relation on $A $ and $a\in A\cdot
A $ then a set $ (a/\Theta) \bigcap (A\cdot A) $ denote by $
\Theta _ {a}$, where $a/\Theta $ is the equivalence  class.

\begin {lemma} \label{the Insert of an idempotent} Let $\langle A,\cdot\rangle$ be Abelian semigroup.
For any idempotent $f \in A $ and all $x, y\in A $ we have $xy=xfy
$.
\end {lemma}

\begin {proof}
Let $x, y \in A $ and $f\in A $ be an idempotent. Then $xf=xff $.
As semigroup $\langle A,\cdot\rangle$ is Abelian we have $xy=xfy
$.
\end {proof}

\begin {lemma} \label{Relations X, Y} Let $\langle A,\cdot\rangle$ be Abelian semigroup satisfying condition $ (*) $. Then

1) the relations $X $ and $Y $ are the equivalence relations on $A
\cdot A $;

2) for all idempotents $e, f\in A $ the equality $ \Phi _ {e}
\bigcap \Psi _ {f} = \{ef \}$ holds and $ef $ is an idempotent;

3) for any $a\in A\cdot A $ there exist the idempotents $e, f \in
A $ such that $a \in X _ {e} \bigcap Y _ {f}$ and $X _ {e} \bigcap
Y _ {f} =Z _ {ef}$.
\end {lemma}

\begin {proof} Let us prove \emph {1)}. We claim that the relation $X $ is the equivalence relation
on $A \cdot A $. Show that $X$ is the reflexive relation. Let
$a=bc $ be any element of $A \cdot A $. On condition $ (*) $
$Ac=Aa $. From $a \in Ac $ it follows that $a \in Aa $, i.e. $a=da
$ for some $d \in A $. Let $e=dd $. Then $ea=dda=da=a $. As
$ddda=da $ then $ (ddd) \Phi d $. Hence $e ^ {2} =dddd=dd=e $.
Therefore $aXa $ for any $a \in A \cdot A $.

 Show that $X$ is the transitive relation. Suppose $aXb $ and $bXc $, where
 $a, b, c\in A \cdot A $, that is $ea=a, eb=b, fb=b, fc=c $ for
 some idempotents $e, f\in A $. Since the semigroup $ \langle A, \cdot
 \rangle $ is Abelian then the equality $eb=fb $  implies $ea=fa=a $. Hence $aXc $. Thus
 the relation $X $ is the equivalence relation on $A \cdot A $.
 It is similarly proved that the relation $Y $ is also
 the relation of the equivalence relation on $A \cdot A $.

Let us prove \emph {2)}. Let $e $, $f $ be idempotents. As $eff=ef
$ and $eef=ef $ then $ (ef) \Phi e $ and $ (ef) \Psi f $, i.e. $ef
\in \Phi _ {e} \bigcap \Psi _ {f}$. Let $g \in \Phi _ {e} \bigcap
\Psi _ {f}$, $g=bc $, $b, c \in A $. Since $ge=ee $ then $bcb (ce)
=b (ce) $, that is $ (bcb) \Phi b $. Hence $bcbc=bc $, i.e. $g ^
{2} =g $, in particular $ef $ is an idempotent. If $g, g ' \in
\Phi _ {e} \bigcap \Psi _ {f}$ then, in view of $g, g ' $ are the
idempotents, we have $g=gg=gg '=g'g' =g ' $. Thus, $ \Phi _ {e}
\bigcap \Psi _ {f} = \{ef \}$.

Let us prove \emph {3)}. Since $X$ is the reflexive relation then
for $a\in A \cdot A $ there is an idempotent $e\in A $ such that
$aXe $. Similarly there is an idempotent $f $ such that $aYf $.
Hence $a\in X_e\cap Y_f $. By Lemma \ref{the Insert of an
idempotent}  $a=ea=(ef)a $ and $a=af=a (ef) $. As $ef $ is an
idempotent then $a\in Z _ {ef}$. The equality $X _ {e} \bigcap Y _
{f} =Z _ {ef}$ is proved.
\end {proof}

\begin {lemma} \label{If a ab that aa rest band of ab gr}
If $\langle A,\cdot\rangle$ is an Abelian semigroup satisfying
condition $ (*) $ then $ \langle A\cdot A, \cdot \rangle $ is a
rectangular band of Abelian groups $ \{\langle Z _ {ef}; \cdot
\rangle \mid e, f  \;\mbox {are idempotents} \}$, and $Z _ {ef '}
=Z _ {ef} \cdot Z _ {e'f '}$ for all idempotents $e, f, e ', f
'\in A $.
\end {lemma}

\begin {proof}
The relation $Z $ is the equivalence relation on $A\cdot A $ as
intersection of two equivalence relations.

Let $a\in A\cdot A $. By Lemma \ref{Relations X, Y} there exists
an idempotent $g $ such that $a\in Z_g $. Hence the semigroup $
\langle A\cdot A, \cdot \rangle $ is an union of semigroups $
\langle Z _ {g}, \cdot \rangle $, where $g $ is an idempotent. Let
$g\in A $ be any idempotent. We will show that $ \langle Z_g,
\cdot \rangle $ is an Abelian group. It is clear that $ \langle
Z_g, \cdot \rangle $ is a semigroup with the identity $g $. Let us
show that the equation $ax=b $ has the solution for all $a, b\in Z
_ {g}$. On condition $ (*) $ $Ag=Abg $. Then $g=xbg $ for some
$x\in A $. So $ag=axbg $ and $a = (ax) b $. Therefore by Theorem
\ref{ab gr s 1} $\langle Z_{g},\cdot \rangle$ is an Abelian group.

Suppose $e, f, e ', f ' $ are idempotents. By Lemma \ref{Relations
X, Y} $fe ' $ is an idempotent, and by Lemma \ref{the Insert of an
idempotent} we have $efe'f '=ef' $. We claim that $Z _ {ef} \cdot
Z _ {e'f '} \subseteq Z _ {ef '}$. If $a\in Z _ {ef}$, $b\in Z _
{e'f '}$ then by Lemmas \ref{the Insert of an idempotent} and
\ref{Relations X, Y} $ab=abe'f'=abee'f'=abef'$ and
$ab=efab=eff'ab=ef'ab$, that is $ab\in Z _ {ef '}$ and $Z _ {ef}
\cdot Z _ {e'f '} \subseteq Z _ {ef '}$.  Thus the semigroup $
\langle A\cdot A, \cdot \rangle $ is a rectangular band of Abelian
groups $ \{\langle Z _ {ef}; \cdot \rangle \mid e, f  \;\mbox {are
idempotents} \}$.
\end {proof}

\begin {theorem} \label{a - ab iff a - inflated of rec band ab gr} Let $\langle A,\cdot\rangle$
be a semigroup satisfying condition $ (*) $. Then $\langle
A,\cdot\rangle$ is an Abelian algebra iff  $\langle
A,\cdot\rangle$ is an inflation of a rectangular band of Abelian
groups and the product of idempotents of $A $ is idempotent of $A
$.
\end {theorem}

\begin {proof}
 Let $\langle A,\cdot\rangle$ be Abelian
semigroup satisfying condition $ (*) $. By Lemma \ref{Relations X,
Y} the product of idempotents of $A $ is an idempotent of $A $. By
Lemma \ref{If a ab that aa rest band of ab gr} semigroup $ \langle
A\cdot A, \cdot \rangle $ is a rectangular band of Abelian groups
$ \{\langle Z _ {ef}; \cdot \rangle \mid e, f \;\mbox {are
idempotents} \}$ and $Z _ {ef '}=Z _ {ef} \cdot Z _ {e'f '}$ for
any idempotents $e, f, e ', f '\in A $.

Let us show that $\langle A,\cdot\rangle$ is an inflation of
rectangular band of Abelian groups. We will fix any element $t\in
A $ not belonging to the set $A\cdot A $. By Remark \ref{Has
inflat - rel of equival} it is enough to find an idempotent $g $
such that $gtg\in t/\alpha $. Let $e $ and $f $ be any
idempotents, $ (te) (ft) \in Z_g $, where $g $ is an idempotent,
and so an identity of group $ \langle Z _ {g}; \cdot \rangle $. By
Lemma \ref{the Insert of an idempotent} $ (te) (ft) = (te ') (f't)
$ for all idempotents $e ', f ' $, therefore $g $ does not depend
on a choice of idempotents $e $ and $f $.

Let us show that $ (gtg) \alpha t $. Since $ (te) (ft) \in Z_g $
then $gt (eft) =t (eft) $. As the semigroup $\langle
A,\cdot\rangle$ is Abelian then $gtg=tg $. Similarly, $gtg=gt $.
So by  Lemma \ref{the Insert of an idempotent} for all $x\in
A\cdot A $  we receive $xt=xgt=xgtg $ and $tx=tgx=gtgx $, i.e. $
(gtg) \alpha t $ and $ \langle A; \cdot \rangle $ is an inflation
of semigroup $ \langle A\cdot A; \cdot \rangle $.

Let us prove the sufficiency. By Lemma \ref{if a inflated then aa
is rectangl ab gr} $\langle A,\cdot\rangle$ is an inflation of
semigroup $ \langle A\cdot A; \cdot \rangle $, being, on the
condition, a rectangular band of Abelian groups $Z _ {i\lambda}$
with identities $e _ {i\lambda}$, where $e _ {i\lambda} \cdot e _
{j\mu} =e _ {i\mu}$ for all $i, j\in I $ and $ \lambda, \mu \in
\Lambda $. We claim that $\langle A;\cdot\rangle$ is an Abelian
algebra. Let $x,y,a_k,b_k\in A$ ($k\in \{1,2\}$) and $a_{1}x
b_{1}=a_{2}x b_{2}$. We will show that $a_{1}y b_{1}=a_{2}y
b_{2}$. By definition of an inflation of semigroup there are
$x'\in Z_{j\mu}$, $y'\in Z_{l\nu}$, $a'_{k}\in Z_{i\lambda _{k}}$,
$b'_{k}\in Z_{i_{k}\lambda}$ ($k\in \{1,2\}$) such that $x\alpha
x'$, $y\alpha y'$, $a_{k}\alpha a'_{k}$, $b_{k}\alpha b'_{k}$.
Hence $a'_{1}x' b'_{1}=a'_{2}x' b'_{2}$. For all $k\in \{1,2\}$
 $$a'_{k}x'b'_{k}=(a'_{k}e_{i\lambda _{k}})x'(e_{i_{k}\lambda } b'_{k})=
a'_{k}(e_{i\lambda}e_{i\lambda _{k}})x'(e_{i_{k}\lambda
}e_{i\lambda}) b'_{k}=$$
 $$=(a'_{k}e_{i\lambda})(e_{i\lambda _{k}}x'e_{i_{k}\lambda
})(e_{i\lambda} b'_{k})=(a'_{k}e_{i\lambda})(e_{i\lambda
_{k}}e_{j\mu}x'e_{j\mu}e_{i_{k}\lambda })(e_{i\lambda} b'_{k})=$$
 $$=(a'_{k}e_{i\lambda})(e_{i\mu}x'e_{j\lambda })(e_{i\lambda}
b'_{k})=a_{k}''x''b_{k}'',$$
 where $a_{k}''=a'_{k}e_{i\lambda}$,
$x''=e_{i\mu}x'e_{j\lambda }$, $b_{k}''=e_{i\lambda} b'_{k}$ и
$a_{k}'', x'', b_{k}''\in Z_{i \lambda}$. So
$a_{1}''x''b_{1}''=a_{2}''x''b_{2}''$. Similarly
$a'_{k}y'b'_{k}=a_{k}''y''b_{k}''$, where
$y''=e_{i\nu}y'e_{l\lambda}\in Z_{i\lambda}$. Since $\langle
Z_{i\lambda};\cdot \rangle$ is an Abelian group then
$a_{1}''y''b_{1}''=a_{2}''y''b_{2}''$, т.е. $a'y'b'=c'y'd'$.
Therefore $ayb=cyd$. Similarly it is proved stationary of
semigroup $\langle A;\cdot \rangle$. Hence by Theorem \ref{ab gr s
1}, полугруппа $\langle A;\cdot \rangle$ is Abelian.
\end {proof}

Note that $ (\ast) $ is used only in the proof of necessity in
Theorem \ref{if a inflated then aa is rectangl ab gr}.

The following \emph {example} shows, that in Theorem \ref{a - ab
iff a - inflated of rec band ab gr} it is impossible to omit the
condition that the product of idempotents is an idempotent.
Consider a set $A =\bigcup \{Z _ {i\lambda} \mid i\in \{0,1 \},
\lambda \in \{0,1 \} \}$, where $ \langle Z _ {i\lambda}; +
\rangle $ are the copies of the residue class group $ \langle
\mathcal {Z} _ {2}; + \rangle $ modulo 2, $ \mathcal {Z} _ {2} =
\{\overline {0}, \overline {1} \}$, $ \overline {0} _ {i\lambda}$
are the copies of an element $ \overline {0}\in Z _ {i\lambda}$, $
\overline {1} _ {i\lambda}$ are the copies of an element $
\overline {1}\in Z _ {i\lambda}$ $ (i\in \{0,1 \}, \lambda \in \{
0,1\}) $. We will extend an operation + on a set $A $ as follows:
\[ \bar{\varepsilon}_{i\lambda}+\bar{\delta}_{j\mu}= \left\{
\begin{array}{ll}
\bar{\varepsilon} _{i\mu}+\bar{\delta}_{i\mu},& \textrm{if } i=j \;\;\textrm{or}\; \lambda =\mu; \\
\bar{\varepsilon} _{i\mu}+\bar{\delta}_{i\mu}+\bar{1}_{i\mu}, &
\textrm{otherwise. }
\end{array} \right. \]

 The semigroup $ \langle A; + \rangle $ is a rectangular band of
 Abelian groups, and the sum of idempotents is not
 an idempotent. For example,
 $ \overline {0} _ {00} + \overline {0} _ {11} = \overline {1} _ {01}$. This semigroup is not
Abelian since $ \overline {0} _ {10} + \overline {1} _ {01} =
\overline {1} _ {11} + \overline {1} _ {01}$, and $ \overline {0}
_ {10} + \overline {1} _ {11} \neq \overline {1} _ {11} +
\overline {1} _ {11}$.

In \cite{war2} it is proved that the semisimple semigroup $\langle
A,\cdot \rangle$ is Abelian iff $\langle A,\cdot \rangle\cong
\langle H,\cdot \rangle\times \langle I,\cdot \rangle\times
\langle J,\cdot \rangle$, where $\langle H,\cdot \rangle$ is an
Abelian group, $\langle I,\cdot \rangle$ is a left zero semigroup,
$\langle J,\cdot \rangle$ is a right zero semigroup. It is not
difficult to prove the following Remark.

\begin{remark}\label{pryam svyazk} A semigroup $\langle A,\cdot \rangle$
is a rectangular band of the Abelian groups and the product of
idempotents of $A $ is idempotent of $A$ iff $\langle A,\cdot
\rangle\cong \langle H,\cdot \rangle\times \langle I,\cdot
\rangle\times \langle J,\cdot \rangle$, where $\langle H,\cdot
\rangle$ is an Abelian group, $\langle I,\cdot \rangle$ is a left
zero semigroup, $\langle J,\cdot \rangle$ is a right zero
semigroup.
\end{remark}

Then by Theorem \ref{a - ab iff a - inflated of rec band ab gr}
and Remark \ref{pryam svyazk} we have

\begin{corollary}
Let $\langle A,\cdot\rangle$ be a semigroup satisfying condition
$(*)$. Then $\langle A,\cdot\rangle$ is an Abelian algebra iff
$\langle A,\cdot\rangle$ is an inflation of a semisimple Abelian
semigroup.
\end{corollary}

The following Proposition give us the sufficient condition for a
Hamilton semigroup.

\begin{proposition}\label{then gam semigr} If the semigroup
$\langle A,\cdot \rangle$ is an inflation of a rectangular band of
the periodic Abelian groups and the product of idempotents of $A $
is idempotent of $A $ then $\langle A,\cdot \rangle$ is a
Hamiltonian algebra.
\end{proposition}

\begin{proof} Let the semigroup
$\langle A,\cdot \rangle$ is an inflation of $\langle C,\cdot
\rangle$ where $\langle C,\cdot \rangle$ is a rectangular band of
the periodic Abelian groups. Suppose that $\langle B,\cdot
\rangle$ is a subsemigroup of $\langle A,\cdot \rangle$. By Remark
\ref{pryam svyazk} and Proposition 4.6 from \cite{war2} $\langle
C,\cdot \rangle$ is a Hamilton semigroup. Then there is the
congruence $\Theta_1$ on $\langle C,\cdot \rangle$ such that
$B\cap C$ is block of the congruence. Denote the least equivalence
relation on $\langle A,\cdot \rangle$ which contains $\Theta_1$
and $\{(a,b)\in A\times A \mid a\alpha b\}$ by $\Theta$. Note that
if $b'\in B$, $b'\alpha b$ and $b\in C$ than $b\in B$. Really
since all elements of semigroup $\langle C,\cdot \rangle$ have
finite order then $(b')^{k}=b^{k}=e$ for some $k\geq 1$, where $e$
is an unit of a group containing $b$. Then
$b=eb=(b')^{k}b=(b')^{k+1}\in B$. So $\Theta$ is a congruence on
$\langle A,\cdot \rangle$ such that $B$ is block of the
congruence.
\end{proof}

\begin{theorem}An Abelian semigroup $\langle A,\cdot \rangle$ is a Hamiltonian algebra iff
$\langle A,\cdot \rangle$ is an inflation of a rectangular band of
the periodic Abelian groups and the product of idempotents of $A $
is idempotent of $A $.
\end{theorem}

\begin{proof} The sufficiency is follows by Proposition \ref{then gam
semigr}. Let us prove necessity. Let $\langle A,\cdot \rangle$ be
Hamiltonian and Abelian semigroup and $a,b\in A$. We claim that
$aA\subseteq abA$. By Lemma \ref{gam polugr period} there exist
$i,j$, $1\leq i<j$, such that $b^i=b^j$ exists. Then
$a(b^{i}c)=a(b^{j}c)$. Hence $a(b^{i}c)=ab^{j-i}(b^{i}c)$. Since
an algebra $\langle A,\cdot \rangle$ is Abelian then
$ac=ab^{j-i}c$, that is $ac \in abA$ and the inclusion
$aA\subseteq abA$ is proved. Therefore the semigroup $\langle
A,\cdot \rangle$ satisfies condition $(*)$. By Theorem \ref{a - ab
iff a - inflated of rec band ab gr}the semigroup $\langle A,\cdot
\rangle$ is an inflation of a rectangular band of Abelian groups,
which are periodic by Lemma \ref{gam polugr period}, and the
product of idempotents of $A $ is idempotent of $A$.
\end{proof}


\begin{thebibliography}{999}

\bibitem[1]{hm} D. Hobby, R. McKenzie \emph{The Structure of Finite
Algebras} // Contemporary Mathematics. V. 76. American
Mathematical Society. Providence. RI. 1988.

\bibitem[2]{kv} E. W. Kiss, M. A. Valeriote \emph{Abelian algebras and the Hamiltonian property} // J. Pure Appl.
Algebra 1993. V.87. No.1. P.37–49.

\bibitem[3] {kval} E.W. Kiss, M. A. Valeriote \emph{Strongly abelian varieties and the
Hamiltonian property} // Canad. J. Math. 1991. V.43. No.2. P.1-16.

\bibitem[4]{ovch} E.V. Ovchinnikova \emph{On Abelian groupoids with image of small power}
// Algebra and Model Theory. Collection of papers.
Novosibirsk State Technical University, 2005.

\bibitem[5]{war1} R.J. Warne \emph{Semigroups obeying the term condition} // Algebra
Universalis. 1994. V.31. No.1. P.113–123.

\bibitem[6]{war2} R.V. Warne \emph{TC semigroups and inflatitions} // Semigroup
Forum. 1997. V.54. No.1. P.271-277.

\bibitem[7]{belous} V.D. Belousov \emph{Foundations of the Theory of Quasigroups and Loops} // M.: Nauka. 1967.

\bibitem[8]{kp} A.H. Clifford, G.B. Preston \emph{The Algebraic Theory
of Semigroup} // No.7. Providence. R.I. 1967.

\end{thebibliography}
\end{document}